\documentclass{lms}
\usepackage{amsmath}
\usepackage{amssymb}
\usepackage[bookmarks]{hyperref}
\usepackage{enumerate}
\usepackage{tikz}

\newtheorem{theorem}{Theorem}[section] 
\newtheorem{lemma}[theorem]{Lemma}     
\newtheorem{corollary}[theorem]{Corollary}
\newtheorem{proposition}[theorem]{Proposition}
\newtheorem{prop}[theorem]{Proposition}

\newnumbered{assertion}{Assertion}    
\newnumbered{conjecture}{Conjecture}  
\newnumbered{definition}{Definition}
\newnumbered{hypothesis}{Hypothesis}
\newnumbered{remark}{Remark}
\newnumbered{note}{Note}
\newnumbered{observation}{Observation}
\newnumbered{problem}{Problem}
\newnumbered{question}{Question}
\newnumbered{algorithm}{Algorithm}
\newnumbered{example}{Example}
\newunnumbered{notation}{Notation} 

\renewcommand{\theenumi}{\roman{enumi}}

\newcommand{\Pol}{\mathcal{P}}
\newcommand{\cX}{\mathcal{X}}
\newcommand{\cM}{\mathcal{M}}

\newcommand{\h}{\mathcal{H}}
\newcommand{\D}{\mathbb{D}}
\newcommand{\T}{\mathbb{T}}
\newcommand{\C}{\mathbb{C}}

\newcommand{\ord}{\operatorname{ord}}
\newcommand{\ro}{\operatorname{ro}}

\newcommand{\inner}[2]{{\langle#1,#2\rangle}}

\newcommand{\binner}[2]{{\big\langle#1,#2\big\rangle}}

\newcommand{\cls}{\operatorname{cl}}
\newcommand{\spn}{\operatorname{span}}

\allowdisplaybreaks[2]

\doi{blms.12620}

\title[]
{Analogues of Finite Blaschke Products as Inner Functions} 

\author{Christopher Felder and Trieu Le}

\classno{30H99 (primary), 30H10 30H20 46E22 (secondary)}

\extraline{The first author was supported in part by NSF Grant \#1565243.}

\begin{document}
\maketitle

\begin{abstract}
We give a generalization of the notion of finite Blaschke products from the perspective of generalized inner functions in various reproducing kernel Hilbert spaces. Further, we study precisely how these functions relate to the so-called Shapiro--Shields functions and  shift-invariant subspaces generated by polynomials. Applying our results, we show that the only entire inner functions on weighted Hardy spaces over the unit disk are multiples of monomials, extending recent work of Cobos and Seco.
\end{abstract}

\section{Introduction and Background}
Let $H^2$ denote the classical Hardy--Hilbert space on the unit disk $\D$, the space of analytic functions on $\D$ that have square-summable Maclaurin coefficients. In 1948, Arne Beurling \cite{beurling1949two} proved several seminal results regarding function and operator theory in $H^2$. One result showed that every function in $H^2$ may be factored as $\theta U$, where $|\theta| =1$ almost everywhere on the unit circle $\mathbb{T}$ (coined as \textit{inner}), and $U$ is such that $\log|U(0)| = \int_0^{2\pi} \log|U(e^{i\theta})|$ (coined \textit{outer}). This result allowed him to characterize shift-invariant subspaces of $H^2$; those subspaces $M \subseteq H^2$ with the property that $SM \subseteq M$, where $S: H^2 \to H^2$ is given by $f(z) \mapsto zf(z)$. Namely, these subspaces must equal $\theta H^2$, for some inner function $\theta$.

The simplest $H^2$-inner functions are called \textit{finite Blaschke products}, given for $\beta_1, \ldots, \beta_n \in \D$ as
\[
B(z) = \lambda \prod_{j = 1}^n \frac{z - \beta_j}{1 - \overline{\beta_j}z},
\]
where $\lambda \in \T$. 
One may check that $|B| = 1$ on $\T$, and so $B$ is in fact inner. We recall that \textit{Blaschke factors}, given when $n=1$ in $B$ above, define the automorphisms of $\D$, up to a unimodular constant. See the recent book \cite{GarciaSC2018} for a nice treatment of finite Blaschke products and their applications.

Let us start with an observation. Applying a partial fractions decomposition, and assuming simple zeros different from the origin, one can check that any finite Blaschke product can be expressed as
\[
B(z) = c_0 - \sum_{j = 1}^n \frac{c_j}{1 - \overline{\beta_j}z}
\]
for some constants $c_0, c_1, \ldots, c_n \in \C$. 
Further, each term in the sum can be seen as a scalar multiple of the Szeg\H{o} kernel,  $s_{\lambda}(z) = 1/(1-\overline{\lambda}z)$, $\lambda \in \D$. 
Noting that $s_0(z) = 1$, we have
\[
B(z) = c_0s_0(z) - \sum_{j = 1}^n c_j s_{\beta_j}(z).
\]
Further, the Szeg\H{o} kernel is the \textit{reproducing kernel} for $H^2$, which means that for every $\lambda \in \D$ and every $f \in H^2$, $\langle f, s_\lambda \rangle_{H^2} = f(\lambda)$. Consequently, every Blaschke product with simple zeros can be seen as a linear combination of reproducing kernels. If $B$ also has repeated zeros, certain derivatives of kernel functions will also be needed in the linear combination. Nonetheless, it turns out this observation actually characterizes finite Blaschke products among inner functions: an $H^2$-inner function $f$ is a finite Blaschke product if and only if $f$ is a linear combination of reproducing kernels and their derivatives.

We will prove this result in a more general setting in Theorem \ref{modified_le}. Toward this generalization, we note that the Hardy space $H^2$ is just one example of a function space where point evaluation can be recovered via its inner product. Spaces of this type are known as \textit{reproducing kernel Hilbert spaces} (RKHS). 
The aim of this paper is to show, working in a general class of RKHSs, and using a generalized definition of inner, that there are analogous inner functions characterized as certain linear combinations of reproducing kernels. We call functions characterized in this way \textit{analogues of finite Blaschke products}. Further, we show precisely how these functions arise as certain Gram determinants, or as certain projections onto shift-invariant subspaces generated by polynomials. 

Throughout this paper, $\h$ will denote a reproducing kernel Hilbert space of analytic functions on some planar domain $\Omega \subset \C$ with $0 \in \Omega$. Formally, $\h$ is a complete inner product space comprised of functions mapping $\Omega \to \C$ that has the reproducing kernel property, i.e. for each $w \in \Omega$, the linear functional given by $f \mapsto f(w)$ is bounded. Equivalently, by the Riesz representation theorem, for each $w \in \Omega$, there exists a unique function $k_w \in \h$ such that for all $f \in \h$, $\langle f, k_w \rangle = f(w)$. We also ask that $\h$ has the following properties:
\begin{enumerate}[(i)]
\item The forward shift operator $S$, given by $f(z) \mapsto zf(z)$, is bounded on $\h$.
\item The analytic polynomials $\Pol$ form a dense subset of $\h$.
\end{enumerate}
The following notation will be used throughout the paper:
\begin{enumerate}[(a)]
    \item When $V \subseteq \h$ is a closed subspace, we will use $\Pi_V: \h \to V$ to denote the orthogonal projection from $\h$ onto $V$.
    \item When $X \subseteq \h$ is a subset, we will use $\cls_{\h}(X)$ to denote the closure of $X$ in $\h$. 
    \item For $f \in \h$, we will use the standard notation $\ord_0(f)$ to denote the order of the zero of $f$ at the origin.
    \item For a polynomial $f$, we will let $Z(f)$ be the multiset containing the zeros $f$, i.e. $Z(f)$ is the zero set of $f$, each zero listed with its multiplicity. 
\end{enumerate}

The primary examples of such spaces we keep in mind are those where $\Omega = \D$ and $\h$ is a subset of the analytic functions on $\D$. A well-known family of such spaces are the so-called weighted Hardy spaces. Given a sequence of positive numbers $w = \{w_k\}_{k \ge 0}$, with $\lim_{k \to \infty} w_k/w_{k+1} = 1$, define the weighted Hardy space as
\[
H^2_w : = \left\{ \sum_{k \ge 0} a_k z^k \in \textrm{Hol}(\D) : \sum_{k \ge 0} w_k |a_k|^2 < \infty \right\}.
\]
For $f(z) = \sum_{k \ge 0} a_k z^k$ and $g(z) = \sum_{k \ge 0} b_k z^k$ in $H^2_w$, their inner product in $H^2_w$ is given by 
\[
\langle f, g \rangle_w = \sum_{k\ge 0}w_k a_k \overline{b_k}. 
\]
One may also verify (e.g. see \cite[Section~2.1]{Cowen1995}) that these spaces are reproducing kernel Hilbert spaces with reproducing kernel  
\[
k_\beta(z) = \sum_{n\ge0} \frac{1}{w_k}\left(\bar{\beta} z\right)^n
\]
for, a priori, $\beta, z \in \D$. It turns out that under appropriate conditions on the weight sequence, the above kernel function extends to points on the unit circle as well. 

If we let $\alpha \in \mathbb{R}$ and $w = \{(k+1)^\alpha\}_{k\ge 0}$, we recover the so-called Dirichlet-type space $\mathcal{D}_\alpha$. Of these spaces, $\alpha = -1$ coincides with the Bergman space $A^2$, $\alpha = 0$ corresponds to the Hardy space $H^2$, and $\alpha = 1$ gives the Dirichlet space $\mathcal{D}$. 

\section{Inner Functions}
Although analogous definitions of Beurling's inner and outer functions in spaces other than $H^2$ have been made, the objects they describe are not as well understood. 
We give these definitions here. 

\begin{definition}[(Cyclic function)]
Say that $f \in \h$ is cyclic (for $S$ in $\h$) if 
\[
[f] : = \cls_\h\left(\textrm{span}\{ z^kf : k \ge 0 \}\right)
\]
is equal to all of $\h$. 
\end{definition}
Note that $[f]$ is always a shift-invariant subspace, i.e. $S[f] \subseteq [f]$, and is the smallest such containing $f$. The space $[f]$ is read ``bracket $f$" or ``the (shift-invariant) subspace generated by $f$." In $H^2$, it is well known that cyclic functions are precisely the outer functions.

Recall that a bounded analytic function $f$ on $\D$ is called an inner function if $|f(\zeta)|=1$ for a.e. $|\zeta|=1$. Inner functions play an important role in function theory and operator theory on Hardy spaces. See \cite{ChalendarNWEJM2015} for a recent survey of classical and new results linking inner functions and operator theory. Besides the above definition, inner functions can also be characterized via the inner product in $H^2$. Indeed, it can be checked that a function $f\in H^2$ is inner if and only if $\|f\|_{H^2}=1$ and $\inner{z^mf}{f}=0$ for all integers $m\geq 1$.

In the case of the Dirichlet space $\mathcal{D}$, Richter \cite{RichterJRAM1988} showed that any shift-invariant subspace is also generated by a single function that satisfies the same orthogonality properties above.
 Aleman, Richter and Sundberg \cite{aleman1996beurling} proved an analogue of Beurling's Theorem for the Bergman space $A^2$: any invariant subspace $\cM$ of $A^2$ is generated by the so-called wandering subspace $\cM\ominus z\cM$. Any unit norm function in this subspace satisfies $\|f\|_{A^2}=1$ and $z^mf\perp f$ for all $m\geq 1$ and is called an $A^2$-inner function. Prior to this work, Hedenmalm \cite{HedenmalmJRAM1991} showed the existence of so-called contractive zero-divisors, which play the role of Blaschke products in the Bergman space. In certain cases, explicit formulas for these functions have been given, e.g. see MacGregor and Stessin \cite{macgregor1994weighted} and Hansbo \cite{hansbo1998reproducing}. These results are phrased in the language of \textit{extremal functions}. Although our work here will not explicitly cover this aspect, it is well known that (normalized) inner functions are solutions to the extremal problem
\[
\sup\left\{ \operatorname{Re}(g^{(d)}(0)) : g \in \mathcal{M}, \|g\| \le 1 \right\},
\]
where $\mathcal{M}$ is a shift-invariant subspace and $d$ is the smallest integer so that $z^d \notin \mathcal{M}^\perp$.
 See \cite[Chapters~5 and 9]{DurenAMS2004} and \cite[Chapter~3]{HedenmalmSpringer2000} for a detailed discussion of inner functions on Bergman spaces $A^p$. 
Researchers have thus defined the notion of inner functions in more general reproducing kernel Hilbert spaces.

\begin{definition}[(Inner function)]
Say that $f \in \h \setminus \{0\}$ is $\h$-inner if, for all $k \ge 1$,
\[
\langle f, z^k f \rangle = 0.
\]
\end{definition}

This definition was originally considered in \cite{aleman1996beurling} for the Bergman space. The authors there also require an inner function to be of unit norm, as well as other authors. In a recent paper, Cheng, Mashreghi and Ross \cite{cheng2019inner} introduced and studied the notion of inner functions with respect to a bounded linear operator. However, they do not require unit norm, which turns out more convenient in several situations. We will follow their approach in this work. 
Although no function-theoretic description of inner functions is known in general reproducing Hilbert spaces, there are known constructions of certain types of inner functions. We will introduce one of these constructions at the end of Section \ref{RepPoints}. 

B\'en\'eteau et al. \cite{BeneteauJLMS2016,BeneteauCMB2018} studied inner functions and examined the connections between them and optimal polynomial approximants on weighted Hardy spaces. They also described a method to construct inner functions that are analogues of finite Blaschke products with simple zeroes. In \cite{SecoCAOT2019}, Seco discussed inner functions on Dirichlet-type spaces and characterized such functions as those whose norm and multiplier norm are equal. 
In \cite{LeAMP2020}, the second author studied inner functions on weighted Hardy spaces and obtained generalizations of several results from \cite{BeneteauCMB2018,SecoCAOT2019}. 
In a recent paper \cite{BeneteauAMP2019}, B\'en\'eteau et al. investigated inner functions on general simply connected domains in the complex plane. We would like to mention that operator-valued inner functions on vector-valued weighted Hardy spaces have also been defined and studied \cite{BallIEOT2013,BallASM2013,OlofssonAiA2007}. In particular, Ball and Bolotnikov \cite{BallASM2013} obtained a realization of inner functions on vector-valued weighted Hardy spaces. In \cite{BallCM2016}, they investigated the expansive multiplier property of inner functions. They obtained a sufficient condition on the weight sequences for which any inner function has the expansive multiplier property. Recently, Cheng, Mashreghi, and Ross considered inner vectors for Toeplitz operators \cite{MR4061942} and for the shift operator in the Banach space setting of $\ell^p_A$ \cite{MR3976584}.

We point out here that if $f$ is inner, then $\langle pf, f \rangle = p(0)\|f\|^2$ for all $p \in \mathcal{P}$, and equivalently, up to a constant, $f$ is the reproducing kernel at the origin for $[f]$ (we ask for $0 \in \Omega$ for this kernel to make sense). The converse is also true, and we record this result below.

\begin{prop}
Let $f \in \h$ and suppose $d = \ord_0(f)$. Then $f$ is $\h$-inner if and only if $f$ is a non-zero constant multiple of $\Pi_{[f]}(k_0^{(d)})$.
Moreover, $\Pi_{[f]}(k_0^{(d)})$ is always an $\h$-inner function. 
\end{prop}
\begin{proof}
Without loss of generality, assume that $f$ has unit norm. Consider the forward implication. For all analytic polynomials $p$, we have
\begin{align*}
    \left\langle pf, \overline{f^{(d)}(0)}f -  \Pi_{[f]}(k_0^{(d)}) \right\rangle
    &= p(0)f^{(d)}(0) - \sum_{k = 0}^d {d \choose k}p^{(k)}(0) f^{(d - k)}(0) \\
    & = p(0)f^{(d)}(0) - p(0)f^{(d)}(0) = 0.
\end{align*}
Thus, $\overline{f^{(d)}(0)}f =  \Pi_{[f]}(k_0^{(d)})$. 

Conversely, suppose $f = c \Pi_{[f]}(k_0^{(d)})$ for some non-zero constant $c$. Then, for $k \ge 1$, we have 
\begin{align*}
    \langle z^kf, f \rangle &= \langle z^kf,  c \Pi_{[f]}(k_0^{(d)})\rangle = \overline{c} \langle z^kf, k_0^{(d)}\rangle = 0
\end{align*}
since $z^kf(z)$ vanishes at the origin with order at least $d+1$.
Thus, $f$ is $\h$-inner. Further, this shows that $\Pi_{[f]}(k_0^{(d)})$ is alway $\h$-inner. 
\end{proof}

In \cite{cheng2019inner}, the authors conducted a robust exploration of inner functions. It was shown there \cite[Proposition~3.1]{cheng2019inner} that every inner function is given by
$\Pi_{[Sf]^{\perp}}(f)$. We show here that, up to a constant, this function is the same as a projection of a kernel onto a shift invariant subspace. 
\begin{prop}
Let $f \in \h$ and let $d = \ord_0(f)$. Put \[J = f - \Pi_{[Sf]}(f) = \Pi_{[Sf]^{\perp}}(f)\] and $v=\Pi_{[f]}(k^{(d)}_0)$.
Then we have
\[v = \frac{v^{(d)}(0)}{f^{(d)}(0)}J.\]
\end{prop}

\begin{proof}
Note that any element of $[Sf]$ vanishes at the origin with multiplicities at least $d+1$. It follows that $k^{(d)}_0\perp [Sf]$ and hence $v\perp [Sf]$, which implies that $v\in [f]\ominus [Sf]$. On the other hand, we also have $J\in [f]\ominus[Sf]$. Since $[f]\ominus[Sf]$ is a one dimensional space, we conclude that $v = \lambda J$ for some constant $\lambda$. To find the constant $\lambda$, let us compute the inner product
\[J^{(d)}(0) = \inner{J}{k^{(d)}_0} = \inner{f}{k^{(d)}_0} - \inner{\Pi_{[Sf]}(f)}{k^{(d)}_0} = f^{(d)}(0)\]
because $k_0^{(d)}\perp [Sf]$. It then follows that
\[\lambda = \frac{\inner{v}{k^{(d)}_0}}{\inner{J}{k^{(d)}_0}} = \frac{v^{(d)}(0)}{f^{(d)}(0)}\]
and the conclusion follows.
\end{proof}

In \cite{shapiro1962zeros}, Shapiro and Shields used Gram determinants to produce linear combinations of reproducing kernels that are Dirichlet-inner functions. This was then generalized by B\'en\'eteau et al. in \cite{BeneteauCMB2018}, and further by the second author \cite{LeAMP2020} to weighted Hardy spaces over the unit disk. Surprisingly, as we will see later, even in general RKHSs in which monomials are not necessarily orthogonal, such a construction (Definition \ref{SS-function}) is the only way to produce inner functions that are linear combinations of kernels (see Theorem \ref{modified_le}).

\subsection{Gram Determinants}\label{Gram}
Let $v_1, \ldots, v_n$ be vectors in an inner product space. We will denote the associated \textit{Gram matrix} by
\[
G(v_1, \ldots, v_n) = (\langle v_i, v_j\rangle)_{1 \le i, j \le n} =
\begin{bmatrix}
\langle v_1, v_1\rangle & \dots & \langle v_1, v_n\rangle\\
\vdots & \ddots & \vdots \\
\langle v_n, v_1\rangle & \dots & \langle v_n, v_n\rangle
\end{bmatrix}.
\]
The \textit{Gram determinant} is then $\det(G(v_1, \ldots, v_n))$. 
Note that for any $x = (x_1, \ldots, x_n) \in \C^n$, 
\[
\langle G(v_1, \ldots, v_n)x, x \rangle = \Big\langle \sum_{i = 1}^n x_i v_i, \sum_{j = 1}^n x_j v_j \Big\rangle \ge 0.
\]
Hence, every Gram matrix is positive semidefinite.
Moreover, the vectors $\{ v_1, \ldots v_n\}$ are linearly independent if and only if $G(v_1, \ldots, v_n)$ has full rank, and equivalently, if and only if $\det(G(v_1, \ldots, v_n)) > 0$.

Similarly, for any vector $u$, we define
\[
D(u; v_1, \ldots, v_n) := \det
\begin{bmatrix}
u & \langle u, v_1\rangle & \dots & \langle u, v_n\rangle\\
v_1 & \langle v_1, v_1\rangle & \dots & \langle v_1, v_n\rangle\\
\vdots & \vdots & \ddots & \vdots \\
v_n &\langle v_n, v_1\rangle & \dots & \langle v_n, v_n\rangle
\end{bmatrix}.
\]
Note that $D(u; v_1, \ldots, v_n)$ is a linear combination of $u, v_1, \ldots, v_n$. 

For our purposes, it is critical to note that $D(u; v_1, \ldots, v_n)$ is orthogonal to each $v_j, 1 \le j \le n$. When $\delta := \det(G(v_1, \ldots, v_n)) > 0$, we have that 
\[
u - \delta^{-1} D(u; v_1, \ldots, v_n)
\]
is a linear combination of the vectors $\{v_1, \ldots, v_n \}$, since the coefficient of $u$ in $D(u; v_1, \ldots, v_n)$ is $\delta$. 
Consequently, we have the following well-known lemma. 
\begin{lemma}\label{SS-proj}
Let $v_1, \ldots, v_n$ be linearly independent vectors in an inner product space $\mathcal{V}$. Then for any $u\in\mathcal{V}$, 
\[
\frac{D(u; v_1, \ldots, v_n)}{\det(G(v_1, \ldots, v_n))}
\]
is the orthogonal projection of $u$ onto $(\spn\{v_1, \ldots, v_n \})^\perp$. 
\end{lemma}

We note here that for any vectors $u, v_1,\ldots, v_n$ with $n\geq 2$, the coefficient of $v_{n}$ in $D(u; v_1,\ldots, v_n)$ is exactly
$-\big\langle D(u; v_1,\ldots, v_{n-1}), v_n\big\rangle$.

For a set of distinct points $\beta_1, \ldots, \beta_n \subset \D$, the authors in \cite{BeneteauCMB2018} coined the \textit{Shapiro--Shields function} as the normalization of $D(1; k_{\beta_1}, \ldots, k_{\beta_n})$, where the inner product is taken in some $H^2_w$. This follows the construction of Shapiro and Shields in \cite{shapiro1962zeros}. They also showed that this function is always inner.

\section{Reproducible Points}\label{RepPoints}
It is important to note that the spaces in which we are working are RKHSs on $\Omega$ but may also have kernels that extend to points outside of $\Omega$. For example, the linear functional of point evaluation extends boundedly to the unit circle in the Dirichlet-type spaces when $\alpha >1$ (more on this below in Example \ref{D-kernels}). 

It is our aim to uncover the relationship between $\h$-inner functions, linear combinations of reproducing kernels, generalized Shapiro--Shields functions, and projections of kernels onto shift-invariant subspaces generated by polynomials. If $\beta \notin \Omega$, but point evaluation is bounded at $\beta$, we can still make sense of, for example, the function $D(u; k_\beta)$ when $u \in \h$. 
We introduce here a framework to handle such points. 

\begin{definition}[(Reproducible point/order)]
Say $\beta \in \C$ is a reproducible point of order $m$ in $\h$ if the linear functional
\[
p \mapsto p^{(m)}(\beta)
\]
extends from $\mathcal{P}$ to a bounded linear functional on $\h$. 
\end{definition}

It is evident that any $\beta\in\Omega$ is reproducible of order $m$ for all $m\geq 1$. On the other hand, points outside of $\Omega$ may only be reproducible up to a certain finite order. First we establish a simple fact about the orders of a reproducible point.

\begin{lemma}\label{reprod}
If $\beta$ is reproducible of order $m$, then it is also reproducible of all orders $0\leq j\leq m$.
\end{lemma}

\begin{proof}
By assumption, there exists a constant $C_m>0$ such that for all polynomials $q$, one has
\[
|q^{(m)}(\beta)| \leq C_m\|q\|.
\]
For any $0\leq j\leq m$, since multiplication by $z$ is bounded, there exists a constant $B_{j}>0$ such that $\|(z-\beta)^{m-j}p\|\leq B_{j}\|p\|$ for all polynomials $p$. On the other hand,
\[
(m-j)!\binom{m}{j}\, p^{(j)}(\beta) = \frac{d^{m}}{dz^{m}}\left((z-\beta)^{m-j}p(z)\right)\Big\rvert_{z=\beta}.
\]
It then follows that
\[
|p^{(j)}(\beta)| \leq \frac{C_{m}B_{j}}{(m-j)!\binom{m}{j}}\|p\|.
\]
Therefore the map $p\mapsto p^{(j)}(\beta)$ extends to a bounded linear functional on $\h$. 
\end{proof}

\begin{definition}
Let $\beta \in \C$ be reproducible of some order. Define the reproducible order of $\beta$ (in $\h$) as 
\[
\ro(\beta) := \sup_m\{ p \mapsto p^{(m)}(\beta) \textrm{ extends as a bounded functional from $\mathcal{P}$ to $\h$}\}.
\]
\end{definition}

\begin{example}[($D_\alpha$, $\alpha >1$)]\label{D-kernels}
Consider the Dirichlet-type spaces $\mathcal{D}_\alpha$. Recall that monomials have norm $\|z^n\|_\alpha = (n+1)^{\alpha/2}$ and that for $|\beta|<1$,
\[
k_\beta(z) = \sum_{n\ge 0} \frac{\bar{\beta}^n}{(n+1)^{\alpha/2}}\frac{z^n}{(n+1)^{\alpha/2}} = \sum_{n\ge 0}\frac{(\bar{\beta}z)^n}{(n+1)^{\alpha}}.
\]
For $\alpha>1$, it is evident that $k_{\beta}$ is a function in $D_{\alpha}$ even for $|\beta|=1$, which implies that all points on the unit circle are reproducible points.

In addition, it is well known that for $|\beta|<1$, the linear functional given by $f \mapsto f^{(m)}(\beta)$ has a reproducing kernel given by
\[
k_\beta^{(m)}(z) = \frac{\partial^m}{\partial \bar{\beta}^m} k_\beta(z) = \sum_{n\geq m}\frac{n(n-1)\cdots (n-m+1)}{(n+1)^{\alpha}}\bar{\beta}^{n-m}\,z^{n}.
\]
The above series expansion shows that $k_{\beta}^{(m)}$ belongs to $D_{\alpha}$ for some (and hence all) $|\beta|=1$ if and only if $\alpha > 2m+1$. As a consequence, all points on the unit circle are reproducible of order $r$ in $D_{\alpha}$, where $r$ is the largest natural number strictly less than $\frac{\alpha-1}{2}$.
\end{example}

\begin{example}[(Local Dirichlet spaces)]
\label{Ex:localD}
Let $\zeta \in \T$ and let $\delta_\zeta$ denote the Dirac measure at $\zeta$. One may form the \textit{local Dirichlet space at $\zeta$} 
\[
\mathcal{D}_{\delta_\zeta}:= \left\{ f \in \textrm{Hol}(\D) : \int_{\D}|f'(z)|^2 \frac{1 - |z|^2}{|z- \zeta|^2}dA(z) < \infty \right\}
\]
where $dA$ is normalized area measure on $\D$. It is well known that the polynomials are dense in $\mathcal{D}_{\delta_\zeta}$ and that it is a reproducing kernel Hilbert space on $\D$. Additionally, $\mathcal{D}_{\delta_\zeta}$ has the property that point evaluation is bounded at $\zeta \in \T$  but not at any other point on $\T$ (see \cite[Theorems 7.2.1 and 8.1.2 (ii)]{el2014primer}).

\end{example}

\begin{example}[($L$ regions)]
Let $\Delta$ be an infinite sequence of disjoint closed discs whose centers lie on the positive real axis and decrease monotonically to zero. By deleting $\Delta$ from $\D$, one obtains an infinitely connected region, known as an $L$ region (see \cite{zalcman1969bounded,mccarthy1994bounded}). 
The origin is a boundary point of the region, and in \cite{mccarthy1994bounded} it was shown that for certain reproducing kernel Hilbert spaces of analytic functions on the region, where the polynomials are dense, point evaluation is bounded at the origin, dependent on the rate of decay of the radii of the disks in $\Delta$. Uniformly perturbing the disks of $\Delta$ to the right by a fixed positive amount $\epsilon > 0$, we obtain an infinitely connected $L$ like region containing zero, where the results describing bounded point evaluation at the origin hold then for the boundary point $\epsilon$. 
We communicate this example to highlight that we need no hypotheses on the connectedness of $\Omega$ and that there is interesting reproducible behavior in this case.
\end{example}

The previous examples show that spaces with bounded point evaluation can behave very differently, and point evaluation can extend outside of the domain $\Omega$ in various ways. 

As we will see, Shapiro--Shields functions can be viewed as projections of a kernel at zero onto certain shift invariant subspaces generated by polynomials. Consequently, we would like to connect reproducibility with the zeros of a polynomial.
\begin{definition}[(Reproducible zeros)]
Let $p \in \mathcal{P}$. The multiset of reproducible zeros of $p$ (in $\h$) is 
\[
R(p) : = \{ \beta \in \C : p^{(m)}(\beta) = 0 \textrm{ and $\beta$ is reproducible of order $m$ in $\h$} \},
\]
listed with multiplicity. 
\end{definition}
Namely, by multiset and ``listed with multiplicity,'' we require that if $p$ has a zero of order $m$ at $\beta$, then $\beta$ appears in $R(p)$ with multiplicity $\min\{\ro(\beta), m\}$. For example, if the point 1 is reproducible of order 2, but not of order 3 (i.e. $\ro(1) = 2$), and the point $\pi i$ is not reproducible in $\h$, then for $p(z) = z(z - 1)^3(z - \pi i)$, we have $R(p) = \{ 0, 1, 1\}$.
Although the multiset of reproducible zeros of a polynomial depends on $\h$, for convenience, we have not included $\h$ in our notation. However, all use of this notation will be clear. 

We would also like to study the multisets of reproducible points that coincide with reproducible zero multisets of polynomials. The following definition allows us to do that.

\begin{definition}[(Reproducible multiset)] 
A finite multiset $Z$ is a reproducible multiset if it can be written as
\[
Z = \Big\{ \underbrace{0 = \beta_0, \ldots, \beta_0}_{m_0 \textrm{ times}}, \underbrace{\beta_1, \ldots, \beta_1}_{m_1 \textrm{ times}}, \ldots, \underbrace{\beta_s, \ldots, \beta_s}_{m_s \textrm{ times}} \Big\}
\]
where each $\beta_j$ is a distinct reproducible point, $\beta_0 = 0$ appears with multiplicity $m_0$ (possibly zero), and for each $1 \le j \le s$, $\beta_j$ appears with at least multiplicity 1, but no multiplicity higher than its reproducible order, i.e. $1 \le m_j \le \ro(\beta_j)$.
\end{definition}

Note that for any polynomial $p$, $R(p)$ is a reproducible multiset. We can now make a full generalization of the Shapiro--Shields function. 

\begin{definition}[(Shapiro--Shields function)]\label{SS-function}
Let $Z$ be a reproducible multiset and put
\[
Z = \Big\{ \underbrace{0 = \beta_0, \ldots, \beta_0}_{m_0 \textrm{ times}}, \underbrace{\beta_1, \ldots, \beta_1}_{m_1 \textrm{ times}}, \ldots, \underbrace{\beta_s, \ldots, \beta_s}_{m_s \textrm{ times}} \Big\}.
\]
The \textit{Shapiro--Shields function} associated to $Z$ is then defined as
\[
\S_Z = D(k_0^{(m_0)}; k_0^{(m_0 - 1)},\ldots, k_0, k_{\beta_1}^{(m_1 - 1)},\ldots, k_{\beta_1},k_{\beta_s}^{(m_s - 1)},\ldots, k_{\beta_s}).
\]
\end{definition}

It is imperative to note that $\S_Z$ vanishes at each $\beta_j$ with multiplicity $m_j$ when $0 \le j \le s$. We would also like to view a Shapiro--Shields function as an orthogonal projection of a reproducing kernel at the origin onto the orthogonal complement of the span of some other kernels. In order to do this, along with use in later applications, we need the following lemma.

\begin{lemma}\label{L:independence_kernels}
Let $\beta_1,\ldots,\beta_s$ be distinct reproducible points of $\h$.
If $m_1,\ldots,m_s$ are non-negative integers such that $m_j \le \ro(\beta_j)$ for $1 \le j \le s$, then the set
\[
\left\{k_{\beta_j}^{(\ell)}:\quad 0\leq\ell\leq m_j, \ 1\leq j\leq s\right\}
\]
is linearly independent in $\h$.
\end{lemma}
\begin{proof}
Suppose $\{c_{j,\ell}:\ 1\leq j\leq s,\ 0\leq\ell\leq m_j\}$ are complex numbers such that
\[
\sum_{j=1}^{s}\sum_{\ell=0}^{m_j}c_{j,\ell}\,k_{\beta_j}^{(\ell)}=0
\]
in $\h$. We need to prove that $c_{j,\ell}=0$ for all such $j$ and $\ell$. It suffices to show $c_{j,m_j}=0$ for all $1 \leq j\leq s$. Fix such an index $j$. Define the polynomial
\[p(z) = (z-\beta_j)^{m_j}\cdot\prod_{t\neq j}(z-\beta_t)^{m_t+1}.\]
Note that $p^{(\ell)}(\beta_t)=0$ for all $0\leq\ell\leq m_t$ and $1\leq t\leq s$ with $t\neq j$. On the other hand, $p^{(m_j)}(\beta_j)\neq 0$ but
\[p^{(\ell)}(\beta_j)=0\text{ for all } 0\leq \ell < m_j.\]
It follows that
\begin{align*}
    \bar{c}_{j,m_j}p^{(m_j)}(\beta_j) & = \binner{p}{\ \sum_{j=1}^{s}\sum_{\ell=0}^{m_j}c_{j,\ell}\,k_{\beta_j}^{(\ell)}} = 0,
\end{align*}
which forces $c_{j,m_j}=0$ because $p^{(m_j)}(\beta_j)\neq 0$.
\end{proof}

In light of Lemma \ref{SS-proj}, the above result tells us that a Shapiro--Shields function $\S_{Z}$, associated to the reproducible multiset
\[
Z = \Big\{ \underbrace{0 = \beta_0, \ldots, \beta_0}_{m_0 \textrm{ times}}, \underbrace{\beta_1, \ldots, \beta_1}_{m_1 \textrm{ times}}, \ldots, \underbrace{\beta_s, \ldots, \beta_s}_{m_s \textrm{ times}} \Big\},
\]
is a nonzero constant multiple of
the projection of $k_0^{(m_0)}$ onto the orthogonal complement of 
\[
\spn\{k_0^{(m_0 - 1)},\ldots, k_0, k_{\beta_1}^{(m_1 - 1)},\ldots, k_{\beta_1},k_{\beta_s}^{(m_s - 1)},\ldots, k_{\beta_s}\}.
\]

\section{Analogues of Finite Blaschke Products}

In this section, we will show that for $p,q \in \Pol$, $[p] = [q]$ if and only if $R(p) = R(q)$. We will then unify the perspective of Shapiro--Shields functions and  projections of kernels at the origin onto shift-invariant subspaces generated by polynomials, giving \textit{analogues of finite Blaschke products}. 

\subsection{Shift Invariant Subspaces}
Note that $k_0^{(d)} - \Pi_{[p]}(k_0^{(d)}) \perp [p]$, so in order to understand $\Pi_{[p]}(k_0^{(d)})$, it is useful to have a characterization of $[p]^\perp$. We do this first when all the zeros of $p$ are reproducible. 

\begin{proposition}
\label{P:invariant_f}
Let $f \in \Pol$ and suppose that 
\[
R(f) = Z(f) = \Big\{ \underbrace{\beta_1, \ldots, \beta_1}_{r_1 \textrm{ times}}, \underbrace{\beta_2, \ldots, \beta_2}_{r_2 \textrm{ times}}, \ldots, \underbrace{\beta_n, \ldots, \beta_n}_{r_n \textrm{ times}} \Big\}.
\]
Then
\[
[f] = \Big(\spn\big\{k_{\beta_j}^{(\ell)}: 0\leq\ell\leq r_j-1,\ 1\leq j\leq n\big\}\Big)^{\perp}.
\]
\end{proposition}

\begin{proof}
Let $\cM$ denote the right hand-side. Then $\cM$ consists of all functions $h$ in $\h$ for which $h^{(\ell)}(\beta_j)=0$ for all $0\leq \ell \leq r_j-1$ and $1\leq j\leq n$. Note that for each $p\in\Pol$, the polynomial $fp$ belongs to $\cM$. It follows that $f\Pol\subseteq\cM$ and hence $[f]\subseteq\cM$, which implies $\cM^{\perp}\subseteq [f]^{\perp}$. On the other hand, since kernel functions are linearly independent by Lemma \ref{L:independence_kernels}, the space $\cM^{\perp}$ is of dimension $d=r_1+\cdots+r_n$. To prove the equality, we only need to show that the dimension of $[f]^{\perp}$ is at most $d$.

We have $f\Pol + \Pol_{d-1} = \Pol$, where $\Pol_{d-1}$ denotes the space of all polynomials of degree at most $d-1$. Taking closure and using the fact that the sum of a closed subspace with a finite dimensional subspace is closed, we have
\[\h = \cls_{\h}(\Pol) = \cls_{\h}(f\Pol+\Pol_{d-1}) = [f] + \Pol_{d-1}.\]
As a result, the dimension of $[f]^{\perp}$ is at most that of $\Pol_{d-1}$, which is $d$. Therefore, we have $\cM^{\perp}=[f]^{\perp}$, which implies $[f]=\cM$ as required.
\end{proof}

This proposition generalizes \cite[Lemma~4.7]{cheng2019inner} where the authors require the zeros of the polynomial to be contained in $\Omega$ and additional properties are imposed on the space.

We will also show that if $f$ has zeros that are not reproducible, this does not change the structure of $[f]$. First though, we need a proposition. 

\begin{proposition}
\label{P:unbd_eval_der}
 Let $\beta$ be a complex number and $m$ be a non-negative integer. Then the following statements hold.
 \begin{enumerate}[(a)]
    \item $\beta$ is not a reproducible point if and only if $(z-\beta)$ is cyclic.
   \item $\beta$ is a reproducible point with $\ro(\beta)\leq m$ if and only if $(z-\beta)$ is not cyclic and \[\cls_{\h}\big((z-\beta)^{m+2}\Pol\big)=\cls_{\h}\big((z-\beta)^{m+1}\Pol\big).\]
 \end{enumerate}
\end{proposition}

\begin{proof}
For any integer $k\geq 0$, define $\cX_k = (z-\beta)^k\Pol$. It is clear that $\cX_{k+1}\subset\cX_{k}$, which shows that the identity $\cls_{\h}(\cX_{k+1})=\cls_{\h}(\cX_{k})$ holds if and only if $\cX_{k+1}$ is dense in $\cX_{k}$ with respect to the norm induced from $\h$. 

On the other hand, define $\Lambda_{k}:\Pol\rightarrow\C$ by $\Lambda_{k}(p) = p^{(k)}(\beta)$ for $p\in\Pol$.
Observe that
\begin{align*}
    \ker(\Lambda_{k}|_{\cX_k}) 
    & = \Big\{(z-\beta)^{k}q(z):\ q\in\Pol\text{ such that } \Lambda_k\big((z-\beta)^kq(z)\big) = 0\Big\}\\
    & = \Big\{(z-\beta)^{k}q(z):\ q\in\Pol\text{ such that } q(\beta)=0\Big\}\\
    & = \cX_{k+1}.
\end{align*}
It follows from a well-known result in functional analysis (e.g., see Proposition 5.2 and Theorem 5.3 in \cite[Chapter~III]{conway1990course}) that $\Lambda_k|_{\cX_k}$ (being a nonzero functional) is unbounded if and only if $\cX_{k+1}=\ker(\Lambda|_{\cX_k})$ is dense in $\cX_k$. 

Therefore, we have just showed that for any $k\geq 0$, $\cls_{\h}(\cX_{k+1}) = \cls_{\h}(\cX_{k})$ if and only if the linear function $\Lambda_k|_{\cX_k}$ is unbounded.

(a) Since $\cX_0=\Pol$ and $\cX_1=(z-\beta)\Pol$, the function $(z-\beta)$ is cyclic if and only if $\cls_{\h}(\cX_1) = \cls_{\h}(\cX_0)$, which, from the argument above, is equivalent to the fact that $\Lambda_0|_{\cX_0}$ is unbounded. Since $\Lambda_0(h)=h(\beta)$ for all $h\in\Pol$, the unboundedness of $\Lambda_0$ means exactly that $\beta$ is not a reproducible point.

(b) {Suppose first $(z-\beta)$ is not cyclic and $\cls_{\h}(\cX_{m+2})=\cls_{\h}(\cX_{m+1})$. Then $\beta$ is a reproducible point} and the linear functional $\Lambda_{m+1}$ is unbounded on $\cX_{m+1}$, hence, unbounded on $\Pol$. This implies that $\beta$ is not reproducible of order $m+1$. That is, $\ro(\beta)\leq m$.

Let us now prove the converse. Suppose that $\beta$ is a reproducible point and $\ro(\beta)\leq m$. By (a), $(z-\beta)$ is not cyclic. To simplify the notation, define $n=\ro(\beta)$. Then the linear functional $\Lambda_{k}$ is bounded for each $0\leq k\leq n$ but $\Lambda_{n+1}$ is unbounded on $\Pol$. We show that actually $\Lambda_{n+1}|_{\cX_{n+1}}$ is unbounded, which implies that $\cls_{\h}(\cX_{n+2}) = \cls_{\h}(\cX_{n+1})$.

Suppose, for the purpose of obtaining a contradiction, that $\Lambda_{n+1}|_{\cX_{n+1}}$ is bounded. For any $h\in\Pol$, we write
\[h = \sum_{0\leq j\leq n}\frac{h^{(j)}(\beta)}{j!}(z-\beta)^{j} + \tilde{h},\] where $\tilde{h}\in\cX_{n+1}$. Then
\[\Lambda_{n+1}(h) = \Lambda_{n+1}|_{\cX_{n+1}}(\tilde{h})\]
and by triangle inequality,
\begin{align*}
    \|\tilde{h}\| 
    & = \Big\|h -\sum_{0\leq j \leq n}\frac{h^{(j)}(\beta)}{j!}(z-\beta)^{j}\Big\| \leq \|h\| + \sum_{0\leq j\leq n}\frac{|h^{(j)}(\beta)|}{j!}\|(z-\beta)^j\|\\
    & \leq \|h\| + \sum_{0\leq j\leq n}\frac{\|\Lambda_j\|}{j!}\|(z-\beta)^j\|\cdot \|h\|.
\end{align*}
Therefore,
\[|\Lambda_{n+1}(h)| = \|\Lambda_{n+1}|_{\cX_{n+1}}(\tilde{h})\| \leq \|\Lambda_{n+1}|_{\cX_{n+1}}\|\cdot\|\tilde{h}\| \leq C\|h\|,\]
which implies that $\Lambda_{n+1}$ is bounded on $\Pol$, a contradiction. 

We have thus showed that $\cls_{\h}(\cX_{n+2}) = \cls_{\h}(\cX_{n+1})$. Now,
\begin{align*}
    \cls_{\h}(\cX_{n+3})  & = \cls_{\h}\Big((z-\beta)\cdot\cls_{\h}(\cX_{n+2})\Big)\\
    & = \cls_{\h}\big((z-\beta)\cdot\cls_{\h}(\cX_{n+1})\big)\\
    & = \cls_{\h}(\cX_{n+2}).
\end{align*}
It then follows inductively that  $\cls_{\h}(\cX_{m+2}) = \cls_{\h}(\cX_{m+1})$.
\end{proof}

The propositions above allows us to provide a complete description of $[f]$ whenever $f$ is a polynomial. 

\begin{theorem}\label{feqfr}
Let $f \in \mathcal{P}$. For each distinct $\beta_j \in R(f)$, let $r_j$ be the multiplicity of $\beta_j$, i.e.
\[
R(f) = \big\{ \underbrace{\beta_1, \ldots, \beta_1}_{r_1 \textrm{ times}}, \underbrace{\beta_2, \ldots, \beta_2}_{r_2 \textrm{ times}}, \ldots, \underbrace{\beta_n, \ldots, \beta_n}_{r_n \textrm{ times}}\big\}.
\]
Then
\[
[f] = \left[ \prod_{\beta \in R(f)} (z - \beta) \right] =  \Big(\spn\big\{k_{\beta_j}^{(\ell)}: 0 \leq \ell \le r_j - 1,\  1 \leq j \leq n \big\}\Big)^{\perp}.
\]
\end{theorem}

\begin{proof}
We first recall the fact that for any multipliers $g,h$ of $\h$, we have
\[
[g\cdot h] = \cls_{\h}(g\cdot[h]).
\]
If $\beta$ is a non-reproducible zero of $f$, then Proposition \ref{P:unbd_eval_der} gives $[(z-\beta)]=\h$. Applying the above identity with $h=z-\beta$ and $g = f/(z-\beta)$, we conclude that $[f]=[f/(z-\beta)]$. So it suffices to consider only the zeros of $f$ with some reproducible order. Put $f(z) = p(z) \prod_{j = 1}^n (z - \beta_j)^{d_j}$, each $\beta_j$ distinct with $d_j \ge r_j$, $1 \le j \le n$, and with the zeros of $p \in \Pol$ being all of the non-reproducible zeros of $f$ (i.e. $p$ is cyclic). Then $[f] = [f/p]$. So without loss of generality, we may assume that $p(z)$ is identically one. Let $h(z) = (z - \beta_1)^{d_1}$. Then by Lemma \ref{P:unbd_eval_der}, we have $[h] = [(z - \beta_1)^{d_1}] = [(z - \beta_1)^{r_1}]$. Letting $g = f/h$, we have 
\begin{align*}
    [f] & = [g\cdot h] = \cls_{\h}(g\cdot[h])\\
& = \cls_{\h}\big(g\cdot[(z-\beta_1)^{d_1}]\big) = [g\cdot (z-\beta_1)^{r_1}]\\
& = \left[\frac{f}{(z-\beta_1)^{d_1 - r_1}}\right].
\end{align*}
Repeating this argument for each $\beta_j$, $2 \le j \le n$, we have 
\begin{align*}
    [f] & = \left[\frac{f}{\prod_{j=1}^{n}(z - \beta_j)^{d_j - r_j}}\right] = \Big[\prod_{j=1}^{n}(z-\beta_j)^{r_j}\Big]  = \left[ \prod_{\beta \in R(f)} (z - \beta) \right]. 
\end{align*}
Applying Proposition \ref{P:invariant_f} then gives the result. 
\end{proof}

This theorem, to which was previously alluded, is a generalization of a result proved for the $H^2_w$ spaces by the first author in \cite[Theorem~5.4]{felder2020general}. Let us illustrate this theorem by applying it to an example in various spaces.  

\begin{example}
Let $f(z)=z^2(z-\frac{i}{2})(z^2-1)^2$. Then the zero multiset of $f$ is
\[Z(f) = \Big\{0,0,\frac{i}{2}, -1, -1, 1, 1\Big\}.\]
Interestingly, by Theorem \ref{feqfr}, the shift-invariant subspace $[f]$ depends greatly on the underlying Hilbert space.

If $\h=D_{\alpha}$ for $\alpha\leq 1$ (which includes the Hardy, Bergman and Dirichlet spaces), then $R(f)=\{0,0,i/2\}$ and
\[[f] = \Big(\spn\big\{k_{0}, k_{0}^{(1)}, k_{i/2}\big\}\Big)^{\perp}.\]
If $\h=D_{\alpha}$ for $1<\alpha\leq 3$, then by Example \ref{D-kernels}, we have $R(f)=\{0,0,i/2,-1,1\}$, which then implies that
\[[f] =\Big(\spn\big\{k_{0}, k_{0}^{(1)}, k_{i/2}, k_{-1}, k_{1}\big\}\Big)^{\perp}.\]
If $\h=D_{\alpha}$ for $\alpha>3$, then by Example \ref{D-kernels} again, $R(f)=Z(f)$ and so
\[[f] = \Big(\spn\big\{k_{0}, k_{0}^{(1)}, k_{i/2}, k_{-1}, k_{-1}^{(1)}, k_{1}, k_{1}^{(1)}\big\}\Big)^{\perp}.\]
On the other hand, if $\h = \mathcal{D}_{\delta_{1}}$, the local Dirichlet space at $1$, then by Example \ref{Ex:localD}, we have $R(f)=\{0,0,i/2,1\}$ and hence,
\[[f] = \Big(\spn\big\{k_{0}, k_{0}^{(1)}, k_{i/2}, k_{1}\big\}\Big)^{\perp}.\]
\end{example}

Theorem \ref{feqfr} also has an immediate and useful corollary. 
\begin{corollary}\label{equal-perp}
Let $p, q \in \Pol$. Then $[p] = [q]$ if and only if $R(p) = R(q)$. 
\end{corollary}
\begin{proof}
The backward implication is given directly by Theorem \ref{feqfr}.
So let $[p] = [q]$ and suppose for contradiction that $R(p) \neq R(q)$. WLOG, there exists $\beta \in R(p)$ with $\beta \notin R(q)$ or $\beta$ having greater multiplicity in $R(p)$ than in $R(q)$. In either case, Theorem \ref{feqfr} implies there is some $n \ge 0$ so that $k_\beta^{(n)} \perp [p] = [q]$. But then $\langle p, k_\beta^{(n)} \rangle = \langle q, k_\beta^{(n)} \rangle = 0$, which is a contradiction, since $\beta \notin R(q)$ or $\beta$ has multiplicity strictly less than $n+1$ in $R(q)$. 
\end{proof}

\subsection{Inner Functions and Linear Combinations of Kernels}
We now show that each inner function that arises as a certain linear combination of reproducing kernels can be identified with a shift invariant subspace and a Shapiro--Shields function. The following theorem generalizes a result of the second author in \cite[Theorem~3.7]{LeAMP2020}, proved initially in the $H^2_w$ spaces. The significance of our contribution is that we do not require monomials be orthogonal and make almost no geometric assumptions on the underlying set for which $\h$ is an RKHS, providing a very general setting for which these results hold. When $\h = H^2$, this result describes classical Blaschke products and in general gives our \textit{analogues of finite Blaschke products}. 

\begin{theorem}\label{modified_le}
Suppose that 
\[
B = \sum_{j=0}^{s}\sum_{\ell=0}^{m_{j}}c_{j,\ell}\,k_{\lambda_j}^{(\ell)}
\]
is an $\h$-inner function, with $c_{j,m_{j}}\neq 0$.
Then $B$ is a constant multiple of $\Pi_{[f]}(k_0^{(d)})$ for some $d$ and some polynomial $f$. Further, $B$ must also be a Shapiro--Shields function. 
\end{theorem}

The function $B$ here is what we call an \textit{analogue of a finite Blaschke product}.

\begin{proof}
Take $B$ as above and without loss of generality, assume that $\|B\|=1$. For any $g \in \Pol$,
\begin{align*}
\inner{g}{k_0} & = g(0) = \inner{gB}{B} = \sum_{j=0}^{s}\sum_{\ell=0}^{m_{j}}\bar{c}_{j,\ell}\inner{gB}{k_{\lambda_j}^{(\ell)}}\\
& = \sum_{j=0}^{s}\sum_{\ell=0}^{m_{j}}\bar{c}_{j,\ell}\sum_{m=0}^{\ell}\binom{\ell}{m}g^{(m)}(\lambda_j)\cdot B^{(\ell-m)}(\lambda_j)\\
    & = \sum_{j=0}^{s}\sum_{m=0}^{m_j}\Big\{\sum_{\ell=m}^{m_j}\bar{c}_{j,\ell}\binom{\ell}{m}B^{(\ell-m)}(\lambda_j)\Big\}g^{(m)}(\lambda_j)\\
    & = \sum_{j=0}^{s}\sum_{m=0}^{m_j}\Big\{\sum_{\ell=m}^{m_j}\bar{c}_{j,\ell}\binom{\ell}{m}B^{(\ell-m)}(\lambda_j)\Big\}\inner{g}{k_{\lambda_j}^{(m)}}.
\end{align*}
Since the set of polynomials is dense in $\h$, we conclude that
\[k_0 = \sum_{j=0}^{s}\sum_{m=0}^{m_j}\Big\{\sum_{\ell=m}^{m_j}{c}_{j,\ell}\binom{\ell}{m}\overline{B^{(\ell-m)}(\lambda_j)}\Big\}k_{\lambda_j}^{(m)}.\]
It then follows from Lemma \ref{L:independence_kernels} that $0\in\{\lambda_0,\ldots,\lambda_s\}$ and for all $j$ and $m$,
\begin{align}
\label{Eqn:inner_condition_B}
    \sum_{\ell=m}^{m_j}\bar{c}_{j,\ell}\binom{\ell}{m}B^{(\ell-m)}(\lambda_j) = \begin{cases}
    0, & \text{ if } \lambda_j\neq 0 \ \text{ or }\ m\geq 1,\\
    1, & \text{ if } \lambda_j = 0 \ \text{ and }\ m=0.
    \end{cases}
\end{align}
Without loss of generality, we shall always assume that $\lambda_0=0$. Then for $j=0$ and $0\leq m\leq m_0$, equation \eqref{Eqn:inner_condition_B} gives
\[\sum_{\ell=m}^{m_0}\bar{c}_{0,\ell}\binom{\ell}{m}B^{(\ell-m)}(0) = 
\begin{cases}
1, & \text{ if } m=0,\\
0, & \text{ if } m\geq 1.
\end{cases}
\]
Since $c_{0,m_0}\neq 0$, we conclude that either $m_0=0$, or $m_0\geq 1$ and $B^{(\ell)}(0)=0$ for all $0\leq\ell\leq m_0-1$. That is, $B\perp\{k_{0}^{(\ell)}: 0\leq\ell\leq m_0-1\}$. 

On the other hand, for $j\neq 0$ and $0\leq m\leq m_j$, by equation \eqref{Eqn:inner_condition_B},
\[\sum_{\ell=m}^{m_j}\bar{c}_{j,\ell}\binom{\ell}{m}B^{(\ell-m)}(\lambda_j) = 0.\]
Since $c_{j,m_j}\neq 0$, it follows that $B^{(\ell)}(\lambda_j)=0$ for all $0\leq\ell\leq m_j$. That is, $B\perp\{k_{\lambda_j}^{(\ell)}: 0\leq\ell\leq m_j\}$. Let $\mathcal{M}$ be the subspace spanned by the functions
\[\big\{k_{0}^{(\ell)}: 0\leq\ell\leq m_0-1\big\}\cup\big\{k_{\lambda_j}^{(\ell)}: 1\leq j\leq s, 0\leq\ell\leq m_j\big\},\] where the first set is considered to be empty if $m_0=0$. Then we have $B = c_{0,m_0}\Pi_{\mathcal{M}^{\perp}}(k_{0}^{(m_0)})$. By Proposition \ref{P:invariant_f}, $\mathcal{M}^{\perp}$ can be recognized as $[f]$, where \[f(z)=z^{m_0}\prod_{j=1}^{s}(z-\lambda_j)^{m_j+1}.\]
As a result, we have shown that $B$ is a projection of $k_0^{(d)}$ onto $[f]$ for some $d$ and some polynomial $f$. Further, by Lemma \ref{SS-proj}, we know that $B$ must also equal the Shapiro--Shields function $\S_{Z(f)}$. 
\end{proof}

\begin{remark}
Recall that finite Blaschke products are the only rational inner functions on the Hardy space (e.g. see \cite[Section~3.5]{GarciaSC2018}). In fact, since any rational function with poles outside of the closed unit disk can be written as a linear combination of $H^2$-kernel functions, this also follows from our Theorem \ref{feqfr}. 

We can also apply our result to investigate rational inner functions on the Bergman space $A^2$. Suppose $\lambda_1,\ldots,\lambda_s$ are distinct nonzero points on the open unit disk and a function of the form
\[B(z) = c_0 + \sum_{j=1}^{s}c_{j}k_{\lambda_j} = c_0 + \sum_{j=1}^{s}\frac{c_j}{(1-\overline{\lambda}_jz)^2}\]
is $A^2$-inner. Theorem \ref{feqfr} shows that $B$ must vanish at all these points. Therefore, there exists a polynomial $q$ of degree at most $s$ such that
\[B(z) = \frac{q(z)\prod_{j=1}^{s}(z-\lambda_j)}{\prod_{j=1}^{s}(1-\overline{\lambda}_jz)^2}.\qquad \]
On the other hand, it follows from the formula for $B$ as a linear combination of Bergman kernels, all residues of $B$ at $1/\overline{\lambda}_1,\ldots, 1/\overline{\lambda}_s$ are zero. As a consequence, $q$ is uniquely determined (up to a constant) from this condition. This approach provides a different way to construct analogues of finite Blaschke products in $A^2$, which does not involve determinants. The above argument can be adapted for the case of repeated points. 
\end{remark}

{In some spaces, we can also deduce when an inner function is an analogue of a finite Blaschke product based on its analyticity. In particular, we can do so when any function that is analytic and non-vanishing on a neighborhood of $\overline{\Omega}$ is also cyclic in $\h$. This assumption is true for all weighted Hardy spaces $H^2_{w}$, including the Hardy, Bergman and Dirichlet-type spaces.}

\begin{corollary}
\label{C:rational_inner_functions}
Assume that any function that is analytic and non-vanishing on a neighborhood of $\overline{\Omega}$ is cyclic in $\h$. Let $f$ be an $\h$-inner function that is analytic on a neighborhood of $\overline{\Omega}$. Then $f$ must be a linear combination of kernel functions and hence, $f$ is determined by Theorem \ref{modified_le}. In particular, all rational $A^2$-inner functions are Shapiro--Shields functions.
\end{corollary}

\begin{proof}
Write $f(z) = p(z)g(z)$, where $p$ is a polynomial and $g$ is analytic and non-vanishing on a neighborhood of $\overline{\Omega}$. By the definition of inner functions, we see that $f$ belongs to the orthogonal complement of $[Sf] = [qg]$, where $q(z)=zp(z)$. Since $g$ is cyclic, it follows that $[Sf]=[q]$. On the other hand, Theorem \ref{feqfr} shows that $[q]^{\perp}$ is spanned by a finite collection of kernel functions. As a result, $f$ is of the form described in Theorem \ref{modified_le}. 
\end{proof}

{We note in the case of the Bergman space, this corollary can be deduced from the characterization of finite zero $A^2$-inner functions due to Duren, Khavinson, and Shapiro \cite{MR1398090}, along with formulae of Stessin and MacGregor \cite[Section~2]{macgregor1994weighted}.}

{We now present another application involving the analyticity of inner functions. Consider a weighted Hardy space $H^2_{w}$, where $w=\{w_k\}_{k\geq 0}$ consists of positive real numbers with \[\lim_{k\to\infty}\frac{w_k}{w_{k+1}}=1.\] In a recent paper, Cobos and Seco \cite{CobosAM2020} showed that under several additional assumptions, the only entire inner functions on $H^2_w$ are multiples of monomials. It turns out that this result remains true on \textit{all} $H^2_w$, as we show here.

We first discuss the regularity of reproducing kernels on $H^2_w$. Recall that if $\beta$ is a reproducible point of order $m$ for $H^2_w$, then the linear functional $f\mapsto f^{(m)}(\beta)$ is given by the reproducing kernel
\[k_{\beta}^{(m)}(z) = \sum_{n\geq m}\frac{n(n-1)\cdots (n-m+1)}{w_n}\bar{\beta}^{n-m}\,z^n.\]
In the case $\beta\neq 0$, we may write
\[k_{\beta}^{(m)}(z) = \sum_{n=0}^{\infty}Q_{\beta,m}(n)\frac{(\bar{\beta}z)^n}{w_n},\] where $Q_{\beta,m}$ is a polynomial of degree $m$ with coefficients depending on $m$ and $\beta$. Due to our assumption on $w$, it is evident that the radius of convergence of $k_{\beta}^{(m)}$ is equal to $1/|\bar{\beta}|$. More generally, for any nonzero points $\beta_1,\ldots,\beta_s$ and nonzero polynomials $Q_1,\ldots,Q_s$, the radius of convergence of the power series
\[\sum_{n=0}^{\infty}\Big(Q_1(n)\bar{\beta}_1^n + \cdots + Q_s(n)\bar{\beta}_s^{n}\Big)\frac{z^n}{w_n}\] is equal to $\min\big\{1/|\bar{\beta}_j|: 1\leq j\leq s\big\}$. To see this, note that with $r=\max\{|\bar{\beta}_j|: 1\leq j\leq s\}$, for certain values of $\delta$, the series
\[\sum_{n=0}^{\infty}\Big(Q_1(n)\bar{\beta}_1^n+\cdots+Q_s(n)\bar{\beta}_s^n\Big)\frac{r^{-n}}{n^{\delta}}\] does not converge, while it does converge for sufficiently large $\delta$. Therefore, \[\limsup_{n\to\infty}\Big|Q_1(n)\bar{\beta}_1^n+\cdots+Q_s(n)\bar{\beta}_s^n\Big|^{1/n} = r.\]
It follows from the discussion above that nonzero linear combination of kernel functions of $H^2_{w}$ cannot be entire unless it is a linear combination of kernels at the origin, hence, is a polynomial. We are now ready to prove a generalization of Cobos and Seco's result.
\begin{proposition}
\label{P:entire_inner}
The only entire $H^2_w$-inner functions are constant multiples of monomials.
\end{proposition}

\begin{proof}
Suppose that $f$ is an $H^2_w$-inner function which is also entire. By Corollary \ref{C:rational_inner_functions}, $f$ is a linear combination of kernel functions and hence, a polynomial. On the other hand, it is immediate from the definition that any $H^2_w$-inner polynomial must be a constant multiple of a monomial. This completes the proof of the proposition.
\end{proof}
}

We now turn back to the general setting, where we can also provide a description of the various ways in which we view analogues of finite Blaschke products. 
\begin{prop}\label{projections}
Let $f \in \mathcal{P}$ and write
\[
R(f) = \Big\{\underbrace{0 = \beta_0, \ldots, \beta_0}_{r_0 \textrm{ times}}, \underbrace{\beta_1, \ldots, \beta_1}_{r_1 \textrm{ times}}, \ldots, \underbrace{\beta_n, \ldots, \beta_n}_{r_n \textrm{ times}} \Big\}
\]
and let $\tilde{f}(z) = \prod_{\beta \in R(f)} (z - \beta)$. 
Then, up to a constant, the following are equal:
\begin{enumerate}
\item $\Pi_{[f]}(k_0^{(r_0)})$
\item $\Pi_{[\tilde{f}]}(k_0^{(r_0)})$
\item The Shapiro--Shields function $\S_{R(f)}$, associated to the reproducible zeros of $f$. 
\item $\varphi := k_0^{(r_0)} - \sum_{j = 0}^{n} \sum_{\ell=0}^{r_j-1} c_{\ell,j}k^{(\ell)}_{\beta_j}$, where the constants $c_{\ell,j}$ are given by $\langle \varphi, k_{\beta_j}^{(\ell)} \rangle = 0$ for $0 \le \ell \le r_j-1$ and $0 \le j \le n$. 
\end{enumerate}

\end{prop}
\begin{proof}
The fact that $\Pi_{[f]}(k_0^{(r_0)}) = \Pi_{[\tilde{f}]}(k_0^{(r_0)})$ is given by Theorem \ref{feqfr}. 
Now, let $\mathcal{M} = \textrm{span}\left\{k_{\beta_j}^{(\ell)}: 0\leq\ell\leq r_j-1,\ 0\leq j\leq n \right\}^{\perp}$. Again, by Theorem \ref{feqfr}, $\mathcal{M} = [f] = [\tilde{f}]$. So by Lemma \ref{SS-proj}, we have $\S_{R(f)}$ and $\Pi_{[\tilde{f}]}(k_0^{(r_0)})$ are equal up to a constant multiple. 
Following the discussion after Definition \ref{SS-function}, we know that $\langle \Pi_{[\tilde{f}]}(k_0^{(r_0)}), k_{\beta_j}^{(\ell)} \rangle=0$ for $0 \le \ell \le r_j-1$ and $0 \le j \le n$, which is also required of $\varphi$. Since $\dim(\mathcal{M}^\perp) = r_0 + \cdots + r_n$, this uniquely determines both $\Pi_{[\tilde{f}]}(k_0^{(r_0)})$ and $\varphi$, which then must be equal up to a constant multiple. 
\end{proof}

We have included the equivalence $(iv)$ above in order to provide a more computationally explicit description. Conditions $(i), (ii)$, and $(iv)$ generalize results of the first author in \cite{felder2020general}, initially proven for $H^2_w$ spaces.

\section{Extraneous Zeros}

So far, we have seen various descriptions of analogues of finite Blaschke products, namely those in Proposition \ref{projections}. We also saw that for $p, q \in \Pol$, $R(p) = R(q)$ if and only if $[p] = [q]$. 
However, in certain settings, we will see that we have a surprising strengthening of this fact. Namely, Corollary \ref{single_vec_proj} will show, under certain assumptions, that $[p] = [q]$ if and only if $\Pi_{[p]}(k_0^{(d)}) = \Pi_{[q]}(k_0^{(d)})$, for appropriate $d$.

In a similar fashion, one may ask, for different reproducible multisets $A$ and $B$, is it possible that $\S_A$ is a constant multiple of $\S_B$? 
We start by pointing out that for a reproducible multiset $Z$, the function $\S_Z$ can vanish off of $Z$ (e.g. see discussion in  \cite[Section~5]{cheng2019inner}). In this case, we naturally say that $\S_Z$ has an extraneous zero. On the other hand, the following lemma asserts that any such extraneous zero must be different from the origin.

\begin{lemma}
\label{L:SS_extraneous_zero}
Let $Z$ be a reproducible multiset with $0$ appearing $m_0$ times. Then $\S_Z$ vanishes at the origin with order precisely $m_0$. Namely, any extraneous zero of a Shapiro--Shields function cannot be at the origin.
\end{lemma}

\begin{proof}
Let $p(z) = \prod_{\beta \in Z}(z - \beta)$. Then the inner product $\inner{\S_{Z}}{k_0^{(m_0)}}$ is a nonzero constant multiple of
\[
    \binner{\Pi_{[p]}(k_0^{(m_0)})}{k_0^{(m_0)}} = \big\|\Pi_{[p]}(k_0^{(m_0)})\big\|^2,
\]
which is not zero since, by hypothesis, $k_0^{(m_0)}$ is not orthogonal to $p$. As a result, any extraneous zero of $\S_Z$ must be different from the origin.
\end{proof}

We will also need the following simple lemma about projection on orthogonal complements of sets of vectors. 
\begin{lemma}
\label{L:orthogonal_proj}
Let $M$ and $N$ be two sets of vectors in a Hilbert space. Then for any $v$, we have $\Pi_{(M\cup N)^{\perp}}(v)=\Pi_{M^{\perp}}(v)$ if and only if $\Pi_{M^{\perp}}(v)\perp N$.
\end{lemma}

In the following proposition, we show that shift-invariant subspaces generated by polynomials are characterized uniquely by the projection of a kernel function at the origin, as long as Shapiro--Shields functions do not possess extraneous reproducible zeroes.

\begin{proposition}\label{extraneous}
The following two statements are equivalent.
\begin{itemize}
    \item[(a)] All Shapiro--Shields functions do not possess extraneous zeros that are reproducible points.
    \item[(b)] For any two polynomials $p$ and $q$, if $\Pi_{[p]}(k_0^{(d_p)})=\Pi_{[q]}(k_0^{(d_q)})$, then $[p]=[q]$. Here $d_p$ (respectively, $d_q$) denotes the multiplicity of the zero of $p$ (respectively, $q$) at the origin.
\end{itemize}
\end{proposition}

\begin{proof}
Assume that for some multiset $Z$ of reproducible points with $\beta_0 = 0$ appearing $m_0 \ge 0$ times and non-zero values $\beta_1, \ldots, \beta_s$ each appearing with multiplicity $m_j$, $1 \le m_j \le \textrm{ord}(\beta_j)$, the corresponding Shapiro--Shields function
\[\S_{Z} = D\big(k_0^{(m_0)}; k_0^{(m_0 - 1)},\ldots, k_0, k_{\beta_1}^{(m_1-1)},\ldots, k_{\beta_1},k_{\beta_s}^{(m_s-1)},\ldots, k_{\beta_s}\big)\]
has an extraneous zero. This means that either $\S_{Z}$ has an additional zero at a reproducible point $\beta\notin Z$ or it has a zero of multiplicity at least $m_j+1$ at some $\beta=\beta_j$. By Lemma \ref{L:SS_extraneous_zero}, $\beta\neq 0$. Applying Proposition \ref{projections}, we note that $\S_{Z}$ is a non-zero constant multiple of $f=\Pi_{[p]}(k_0^{(m_0)})$, where
\[p(z) = \prod_{j=0}^{s}(z-\beta_j)^{m_j}.\]
It follows that $f(\beta)=0$ (so $f\perp k_{\beta}$) in the case $\beta\notin Z$, or $f^{(m_j)}(\beta_j)=0$ (so $f\perp k_{\beta_j}^{(m_j)}$) in the case $\beta=\beta_j$ for some $j$. Put $q(z) = (z-\beta)p(z)$. Then $f$ coincides with $\Pi_{[q]}(k_0^{(m_0)})$, by Proposition \ref{P:invariant_f} and Lemma \ref{L:orthogonal_proj}. Since $\beta\neq 0$, we see that $d_q = d_p = m_0$ and hence $\Pi_{[p]}(k_0^{d_p}) = \Pi_{[q]}(k_0^{d_q})$ even though $[p]\neq [q]$. 

Conversely, assume that (a) holds. Let $p, q$ be two polynomials. Using Theorem \ref{feqfr}, we write
\[[p] = \Big\{k_{\beta_j}^{(\ell)}:  0\leq\ell < m_j,\ 0\leq j\leq s\Big\}^{\perp},\]
and
\[[q] = \Big\{k_{\alpha_j}^{(\ell)}:  0\leq\ell < n_j,\ 0\leq j\leq t\Big\}^{\perp}.\]
We know that $\Pi_{[p]}(k_0^{(d_p)})$ is a non-zero constant multiple of the Shapiro--Shields function corresponding to a multiset $Z_1$ consisting of $\beta_j$ with multiplicity $m_j$ for $0\leq j\leq s$. In the expansion of $\Pi_{[p]}(k_0^{(d_q)})$, for each $1\leq j\leq s$, the coefficient of $k_{\beta_j}^{(m_j-1)}$ is plus or minus of $\inner{\S_{\tilde{Z}_1}}{k_{\beta_j}^{(m_j-1)}}$, where $\tilde{Z}_1$ is obtained from $Z_1$ by removing one appearance of $\beta_j$. By the hypothesis, this coefficient is nonzero. By the same argument, the coefficient of any $k_{\alpha_j}^{(n_j-1)}$ in the expansion of $\Pi_{[q]}(k_0^{(d_q)})$ is nonzero. It follows that if $\Pi_{[p]}(k_0^{(d_p)})=\Pi_{[q]}(k_0^{(d_q)})$, then due to the linear independence of the kernel functions, each $\beta_j$ equals some $\alpha_{\ell}$ and $m_j = n_{\ell}$, and vice versa. Therefore, $[p]=[q]$.
\end{proof}

In general, the projection of a single vector onto a subspace does not characterize the subspace itself. However, in the setting of Theorem \ref{extraneous}, we have a surprising immediate corollary.

\begin{corollary}\label{single_vec_proj}
Let $p,q \in \Pol$ with $\ord_0(p) = d_p$ and $\ord_0(q) = d_q$.  Suppose that $\Pi_{[p]}(k_0^{(d_p)})$ and $\Pi_{[q]}(k_0^{(d_q)})$ vanish only on $R(p)$ and $R(q)$, respectively. Then, $[p] = [q]$ if and only if $\Pi_{[p]}(k_0^{(d_p)}) = \Pi_{[q]}(k_0^{(d_q)})$. 
\end{corollary}

When Shapiro--Shields functions do possess extraneous zeros, we show one way in which two different reproducible multisets can generate the same Shapiro--Shields function. 

\begin{theorem}
Let $A$ be a reproducible multiset. Then for any reproducible multiset $B$ such that $A\subset B\subset Z(\S_A)$, the function $\S_B$ is a constant multiple of $\S_{A}$.
\end{theorem}

\begin{proof}
It suffices to prove the case $B$ having exactly one element more than $A$. Write
\[A = \Big\{\underbrace{0 = \beta_0, \ldots, \beta_0}_{r_0 \textrm{ times}}, \underbrace{\beta_1, \ldots, \beta_1}_{r_1 \textrm{ times}}, \ldots, \underbrace{\beta_n, \ldots, \beta_n}_{r_n \textrm{ times}} \Big\},
\]
where $r_0\geq 0$ and $r_1,\ldots, r_n\geq 1$. Then $\S_A$ is a nonzero constant multiple of the projection $\Pi_{M^{\perp}}(k_0^{(r_0)})$, where
\[M = \{k_{\beta_j}^{(\ell)}: 0\leq\ell< r_j,\ 0\leq j\leq n\}.\]
Let $\beta$ be the extra element that $B$ has more than $A$. Note that $\beta\neq 0$ by Lemma \ref{L:SS_extraneous_zero}. If $\beta=\beta_{s}$ for some $1\leq s\leq n$, set $u = k_{\beta_s}^{(r_s)}$. If $\beta\neq\beta_j$ for all $j$, let $u=k_{\beta}$. Since $B\subset Z(\S_A)$, we have that $\Pi_{M^{\perp}}(k_0^{(r_0)})\perp u$. As a consequence,
\[\Pi_{M^{\perp}}(k_0^{(r_0)}) = \Pi_{(M\cup\{u\})^{\perp}}(k_0^{(r_0)}).\]
On the other hand, $\S_B$ is a nonzero constant multiple of  $\Pi_{(M\cup\{u\})^{\perp}}(k_0^{(r_0)})$. The conclusion of the theorem then follows.
\end{proof}

We conclude with a few remarks. For $f$ a polynomial, we have described $[f]$ in general reproducing kernel Hilbert spaces of analytic functions. As we have seen, $[f]$ depends closely on reproducible points of the underlying space. It would be interesting to go beyond the polynomial case but it seems to be a more difficult problem.

Blaschke products (and more generally, inner functions) on the Hardy space possess a remarkable multiplicative property: if $B_1$ and $B_2$ are finite Blaschke products, then $B_1B_2$ is also a finite Blaschke product. On the other hand, this property does not hold on other spaces, such as the Bergman and Dirichlet spaces, which can be seen from  direct calculation. While the multiplicative property of classical Blaschke products trivially follows from the definition, it is quite curious why such a property even holds true, in light of Theorem  \ref{modified_le}.
In terms of extremal problems, it is somehow remarkable that the solution to the linear $n$-point extremal problem is given by the \textit{product} of the $n$ individual 1-point problems.

Lastly, we would like to mention that analyticity is not actually needed to establish the results throughout this paper. Instead, one could make the weaker assumption that functions in $\h$ are simply smooth at the origin. We have chosen to require analyticity because it is unclear if there are any (interesting) RKHSs for which the polynomials are dense and the shift is bounded, yet not comprised of analytic functions.

\begin{acknowledgements}\label{ackref}
The first author would like to thank John McCarthy and Dima Khavinson for helpful discussion. The authors would like to thank the referee for a thorough reading and many helpful suggestions that improved the presentation of the paper.
\end{acknowledgements}


\begin{thebibliography}{10}

\bibitem{aleman1996beurling}
A.~Aleman, S.~Richter, and C.~Sundberg, \emph{Beurling's theorem for the
  {B}ergman space}, Acta Math. \textbf{177} (1996), no.~2, 275--310.
  \MR{1440934}

\bibitem{BallIEOT2013}
J.A.~Ball and V.~Bolotnikov, \emph{Weighted {B}ergman spaces: shift-invariant
  subspaces and input/state/output linear systems}, Integral Equations Operator
  Theory \textbf{76} (2013), no.~3, 301--356. \MR{3065298}

\bibitem{BallASM2013}
J.A.~Ball and V.~Bolotnikov, \emph{Weighted {H}ardy spaces: shift invariant and coinvariant
  subspaces, linear systems and operator model theory}, Acta Sci. Math.
  (Szeged) \textbf{79} (2013), no.~3-4, 623--686. \MR{3134507}

\bibitem{BallCM2016}
J.A.~Ball and V.~Bolotnikov, \emph{On the expansive property of inner functions in weighted {H}ardy
  spaces}, Complex analysis and dynamical systems {VI}. {P}art 2, Contemp.
  Math., vol. 667, Amer. Math. Soc., Providence, RI, 2016, pp.~47--61.
  \MR{3511251}

\bibitem{BeneteauCMB2018}
C.~B\'{e}n\'{e}teau, M.~Fleeman, D.~Khavinson, D.~Seco, and A.~Sola,
  \emph{Remarks on inner functions and optimal approximants}, Canad. Math.
  Bull. \textbf{61} (2018), no.~4, 704--716. \MR{3846742}

\bibitem{BeneteauAMP2019}
C.~B\'{e}n\'{e}teau, M.~Fleeman, D.~Khavinson, and A.~Sola, \emph{On the
  concept of inner function in {H}ardy and {B}ergman spaces in multiply
  connected domains}, Anal. Math. Phys. \textbf{9} (2019), no.~2, 839--866.
  \MR{3977672}

\bibitem{BeneteauJLMS2016}
C.~B\'{e}n\'{e}teau, D.~Khavinson, C.~Liaw, D.~Seco, and A.~Sola,
  \emph{Orthogonal polynomials, reproducing kernels, and zeros of optimal
  approximants}, J. Lond. Math. Soc. (2) \textbf{94} (2016), no.~3, 726--746.
  \MR{3614926}

\bibitem{beurling1949two}
A.~Beurling, \emph{On two problems concerning linear transformations in
  {H}ilbert space}, Acta Math. \textbf{81} (1948), 239--255. \MR{27954}

\bibitem{ChalendarNWEJM2015}
I.~Chalendar, P.~Gorkin, and J.R.~Partington, \emph{Inner functions and
  operator theory}, North-West. Eur. J. Math. \textbf{1} (2015), 7--22.
  \MR{3417418}

\bibitem{MR3976584}
R.~Cheng, J.~Mashreghi, and W.T. Ross, \emph{Inner functions and zero sets for
  {$\ell^p_A$}}, Trans. Amer. Math. Soc. \textbf{372} (2019), no.~3,
  2045--2072. \MR{3976584}

\bibitem{cheng2019inner}
R.~Cheng, J.~Mashreghi, and W.T. Ross, \emph{Inner functions in reproducing kernel spaces}, Analysis of
  operators on function spaces, Trends Math., Birkh\"{a}user/Springer, Cham,
  2019, pp.~167--211. \MR{4019469}

\bibitem{MR4061942}
R.~Cheng, J.~Mashreghi, and W.T. Ross, \emph{Inner vectors for {T}oeplitz operators}, Complex analysis and
  spectral theory, Contemp. Math., vol. 743, Amer. Math. Soc., Providence, RI,
  2020, pp.~195--212. \MR{4061942}

\bibitem{CobosAM2020}
A.~Cobos and D.~Seco, \emph{No entire inner functions}, Anal. Math. \textbf{46}
  (2020), no.~1, 39--45. \MR{4064578}

\bibitem{conway1990course}
J.B.~Conway, \emph{A course in functional analysis}, second ed., Graduate Texts
  in Mathematics, vol.~96, Springer-Verlag, New York, 1990. \MR{1070713}

\bibitem{Cowen1995}
C.C.~Cowen and B.D.~MacCluer, \emph{Composition operators on spaces of analytic
  functions}, Studies in Advanced Mathematics, CRC Press, Boca Raton, FL, 1995.
  \MR{1397026}

\bibitem{MR1398090}
P.~Duren, D.~Khavinson, and H.S.~Shapiro, \emph{Extremal functions in invariant
  subspaces of {B}ergman spaces}, Illinois J. Math. \textbf{40} (1996), no.~2,
  202--210. \MR{1398090}

\bibitem{DurenAMS2004}
P.~Duren and A.~Schuster, \emph{Bergman spaces}, Mathematical Surveys and
  Monographs, vol. 100, American Mathematical Society, Providence, RI, 2004.
  \MR{2033762}

\bibitem{el2014primer}
O.~El-Fallah, K.~Kellay, J.~Mashreghi, and T.~Ransford, \emph{A primer on the
  {D}irichlet space}, Cambridge Tracts in Mathematics, vol. 203, Cambridge
  University Press, Cambridge, 2014. \MR{3185375}

\bibitem{felder2020general}
C.~Felder, \emph{General optimal polynomial approximants, stabilization, and
  projections of unity}, arXiv preprint arXiv:2003.10015 (2020).

\bibitem{GarciaSC2018}
S.R.~Garcia, J.~Mashreghi, and W.T.~Ross, \emph{Finite {B}laschke products and
  their connections}, Springer, Cham, 2018. \MR{3793610}

\bibitem{hansbo1998reproducing}
J.~Hansbo, \emph{Reproducing kernels and contractive divisors in {B}ergman
  spaces}, Zap. Nauchn. Sem. S.-Peterburg. Otdel. Mat. Inst. Steklov. (POMI)
  \textbf{232} (1996), no.~Issled. po Line\u{\i}n. Oper. i Teor. Funktsi\u{\i}.
  24, 174--198, 217. \MR{1464433}

\bibitem{HedenmalmJRAM1991}
H.~Hedenmalm, \emph{A factorization theorem for square area-integrable analytic
  functions}, J. Reine Angew. Math. \textbf{422} (1991), 45--68. \MR{1133317}

\bibitem{HedenmalmSpringer2000}
H.~Hedenmalm, B.~Korenblum, and K.~Zhu, \emph{Theory of {B}ergman spaces},
  Graduate Texts in Mathematics, vol. 199, Springer-Verlag, New York, 2000.
  \MR{1758653}

\bibitem{LeAMP2020}
T.~Le, \emph{Inner functions in weighted {H}ardy spaces}, Anal. Math. Phys.
  \textbf{10} (2020), no.~2, Paper No. 25, 21. \MR{4108538}

\bibitem{macgregor1994weighted}
T.H.~MacGregor and M.I.~Stessin, \emph{Weighted reproducing kernels in
  {B}ergman spaces}, Michigan Math. J. \textbf{41} (1994), no.~3, 523--533.
  \MR{1297706}

\bibitem{mccarthy1994bounded}
J.E.~McCarthy and L.M.~Yang, \emph{Bounded point evaluations on the boundaries
  of {$L$} regions}, Indiana Univ. Math. J. \textbf{43} (1994), no.~3,
  857--883. \MR{1305950}

\bibitem{OlofssonAiA2007}
A.~Olofsson, \emph{Operator-valued {B}ergman inner functions as transfer
  functions}, Algebra i Analiz \textbf{19} (2007), no.~4, 146--173.
  \MR{2381937}

\bibitem{RichterJRAM1988}
S.~Richter, \emph{Invariant subspaces of the {D}irichlet shift}, J. Reine
  Angew. Math. \textbf{386} (1988), 205--220. \MR{936999}

\bibitem{SecoCAOT2019}
D.~Seco, \emph{A characterization of {D}irichlet-inner functions}, Complex
  Anal. Oper. Theory \textbf{13} (2019), no.~4, 1653--1659. \MR{3957006}

\bibitem{shapiro1962zeros}
H.S.~Shapiro and A.L.~Shields, \emph{On the zeros of functions with finite
  {D}irichlet integral and some related function spaces}, Math. Z. \textbf{80}
  (1962), 217--229. \MR{145082}

\bibitem{zalcman1969bounded}
L.~Zalcman, \emph{Bounded analytic functions on domains of infinite
  connectivity}, Trans. Amer. Math. Soc. \textbf{144} (1969), 241--269.
  \MR{252665}

\end{thebibliography}

\providecommand{\MR}{\relax\ifhmode\unskip\space\fi MR }
\providecommand{\MRhref}[2]{%
  \href{http://www.ams.org/mathscinet-getitem?mr=#1}{#2}
}
\providecommand{\href}[2]{#2}

\affiliationone{
   Christopher Felder\\
   Department of Mathematics and Statistics\\
   Washington University in St. Louis\\
   St. Louis, MO 63136\\
   USA
   \email{cfelder@wustl.edu}}
\affiliationtwo{
   Trieu Le\\
   Department of Mathematics and Statistics\\
   The University of Toledo\\
   Toledo, OH 43606\\
   USA
   \email{trieu.le2@utoledo.edu}}
\end{document}